\documentclass[a4paper,reqno]{amsart}
\usepackage{amssymb, amsmath, amscd}
\usepackage{wrapfig}
\usepackage{color}

\newtheorem{theorem}{Theorem}[section]
\newtheorem{lemma}[theorem]{Lemma}
\newtheorem{remark}[theorem]{Remark}
\newtheorem{definition}[theorem]{Definition}

\makeatletter
 \@addtoreset{equation}{section}
\makeatother\begin{document}
\title{Moduli spaces of  Type~$\mathcal{B}$ surfaces with torsion}
\author{Peter Gilkey}
\address{Mathematics Department, University of Oregon, Eugene, Oregon 97403 USA}
\email{gilkey@uoregon.edu}
\begin{abstract}
We examine moduli spaces of locally homogeneous
surfaces of Type~$\mathcal{B}$ with torsion where the symmetric
Ricci tensor is non-degenerate. We also determine the space of affine Killing vector fields in this context.\end{abstract}
\maketitle
\section{Introduction}
Let $\nabla$ be a connection on the tangent bundle of a smooth manifold $M$ of dimension $m$.
Let $\vec x=(x^1,\dots,x^m)$ be a system of local coordinates on $M$.
Adopt the {\it Einstein convention} and sum over repeated indices to expand $\nabla_{\partial_{x^i}}\partial_{x^j}=\Gamma_{ij}{}^k\partial_{x^k}$ 
where $\Gamma=(\Gamma_{ij}{}^k)$ are the Christoffel symbols
of the connection. We say that $\nabla$ is {\it torsion free} if
$\nabla_XY-\nabla_YX-[X,Y]=0$ or, equivalently, if $\Gamma_{ij}{}^k=\Gamma_{ji}{}^k$. The importance of the torsion free condition lies in the 
observation that $\mathcal{M}$ is torsion free if and only if for every point $P$ of $M$, there exist coordinates centered at $P$ so that $\Gamma_{ij}{}^k(P)=0$.
Thus in the torsion free setting, one may normalize the coordinate system so that only the second and higher order derivatives of the connection $1$-form play a role
in defining local invariants of Weyl type. 
We also note that if $\nabla$ is a connection with torsion, then there is a naturally associated torsion free connection $\tilde\nabla$
with the same parametrized geodesics.  For these reasons, torsion free connections have been studied extensively in the literature.
There are, however, natural situations in which manifolds with torsion enter.  We refer,
for example, to work on torsion-gravity
 \cite{BM16,C12,F13,F14,F15,LP13,LP13a,VCF15}, on hyper-K\"ahler with torsion supersymmetric sigma models
 \cite{FIS14,FS16,FS15,Sx12},
  on string theory \cite{GMW04,I04},
 on almost hypercomplex geometries \cite{M11}, on spin geometries \cite{K10},  on B-metrics \cite{M12,S12}, 
on contact geometries \cite{ACF05}, on almost product manifolds \cite{Me12}, on non-integrable
geometries \cite{ACFH15,B15}, on the non-commutative residue
for manifolds with boundary \cite{WWY14},
on Hermitian and anti-Hermitian
geometry \cite{MG11}, on CR geometry \cite{DL15}, on
Einstein--Weyl gravity at the linearized level \cite{DEG13}, on Yang-Mills flow with torsion \cite{GDLS14}, on ESK theories \cite{RFSC13},
on double field theory \cite{HZ13}, on BRST theory \cite{FLM16},
and on the symplectic and elliptic geometries of gravity \cite{CSE11}. 

The curvature operator $R(X,Y):=\nabla_X\nabla_Y-\nabla_Y\nabla_X-\nabla_{[X,Y]}$ and the Ricci tensor 
$\rho:=\operatorname{Tr}(Z\rightarrow R(Z,X)Y)$ have components:
\begin{eqnarray*}
&&R_{ijk}{}^l=\partial_{x_i}\Gamma_{jk}{}^l-\partial_{x^j}\Gamma_{ik}{}^l
+\Gamma_{in}{}^l\Gamma_{jk}{}^n-\Gamma_{jn}{}^l\Gamma_{ik}{}^n,\\
&&\rho_{jk}=\partial_{x_i}\Gamma_{jk}{}^i-\partial_{x^j}\Gamma_{ik}{}^i
+\Gamma_{in}{}^i\Gamma_{jk}{}^n-\Gamma_{jn}{}^i\Gamma_{ik}{}^n\,.
\end{eqnarray*}
Since in this quite general setting, unlike the pseudo-Riemannian context, the Ricci tensor need not be symmetric, we introduce the {\it symmetric Ricci tensor} defining:
$$
\rho_s(X,Y):=\textstyle\frac12\{\rho(X,Y)+\rho(Y,X)\}\,.
$$

One is interested in classifying such structures up to isomorphism, i.e. up to the action of the pseudo group of germs of diffeomorphisms.
And it is natural to start with the locally homogeneous examples; $\mathcal{M}:=(M,\nabla)$ is said to be
{\it locally homogeneous} if given any two points $P$ and $Q$ of $M$, there is the germ of a diffeomorphism
$\Phi_{P,Q}$ taking $P$ to $Q$ which commutes with $\nabla$.  The case of surfaces is particularly tractable and is of interest in its own
right; connections on surfaces have been used to construct
new examples of pseudo-Riemannian metrics that do not have a Riemannian analog \cite{CGGV09, CGV10, De11, KoSe}.
We shall assume henceforth that $\mathcal{M}$ is not flat. The following result was established in the torsion free
setting by  Opozda \cite{Op04} and later extended by Arias-Marco and Kowalski \cite{AMK08} to the case of surfaces with torsion.
\begin{theorem}
Let $\mathcal{M}=(M,\nabla)$ be a locally homogeneous surface. Then at least one of the following
three possibilities hold which describe the local geometry:
\begin{itemize}
\item[($\mathcal{A}$)] There exists a coordinate atlas so the Christoffel symbols
$\Gamma_{ij}{}^k$ are constant.
\item[($\mathcal{B}$)] There exists a coordinate atlas so the Christoffel symbols have the form
$\Gamma_{ij}{}^k=(x^1)^{-1}C_{ij}{}^k$ for $C_{ij}{}^k$ constant and $x^1>0$.
\item[($\mathcal{C}$)] $\nabla$ is the Levi-Civita connection of a metric of constant Gauss
curvature.
\end{itemize}\end{theorem}
These classes are not disjoint. While there are no surfaces which are both Type~$\mathcal{A}$ and Type~$\mathcal{C}$, there are surfaces which are both
Type~$\mathcal{A}$ and Type~$\mathcal{B}$, and, up to isomorphism, there are are two surfaces which are Type~$\mathcal{B}$ and Type~$\mathcal{C}$ -- the
hyperbolic plane and the Lorentzian analogue. The remaining Type~$\mathcal{C}$ geometry is modeled on that
of the round sphere and plays no role in our analysis. 

As the Type~$\mathcal{A}$ setting was discussed in \cite{BGGP16,BGGP16a,G16}, we shall concentrate
on the Type~$\mathcal{B}$ setting in this paper. We introduce some notational conventions:
\begin{definition}\rm
Let $\mathcal{W}_{\mathcal{B}}(2):=(\mathbb{R}^2)^*\otimes(\mathbb{R}^2)^*\otimes\mathbb{R}^2$; two indices are
down and one index is up. If $C\in\mathcal{W}_{\mathcal{B}}(2)$, let $\Gamma_{ij}{}^k=(x^1)^{-1}C_{ij}{}^k$ be the Christoffel symbols of the
associated connection $\nabla=\nabla^C$ of Type~$\mathcal{B}$. If $p+q=2$, let
$$
\mathcal{W}_{\mathcal{B}}(p,q):=\{C\in\mathcal{W}_{\mathcal{B}}(2):\operatorname{signature}(\rho_{s,C})=(p,q)\}\,.
$$
Let $\mathfrak{W}_{\mathfrak{B}}(p,q)$ (resp. $\mathfrak{W}_{\mathfrak{B}}^+(p,q)$)
be the associated moduli space where we identify two connections $\nabla$ and $\tilde\nabla$ if
there exists a local diffeomorphism (resp. orientation preserving local diffeomorphism) of $\mathbb{R}^+\times\mathbb{R}$ intertwining $\tilde\nabla$ and $\nabla$.
Let $\rho_{C}$ be the Ricci tensor of $\nabla^C$.
The components of $\rho_C$ are then given by:
\begin{equation}\begin{array}{l}\label{E1.a}
 \rho_{C;11}=(x^1)^{-2}\{(C_{11}{}^1-C_{12}{}^2+1) C_{21}{}^2+C_{11}{}^2 (C_{22}{}^2-C_{21}{}^1)\},\\[0.05in]
 \rho_{C;12}=(x^1)^{-2}\{C_{22}{}^2+C_{12}{}^1 C_{21}{}^2-C_{11}{}^2 C_{22}{}^1\}, \\[0.05in]
 \rho_{C;21}=(x^1)^{-2}\{-C_{21}{}^1+C_{12}{}^1 C_{21}{}^2-C_{11}{}^2 C_{22}{}^1\},\\[0.05in]
 \rho_{C;22}=(x^1)^{-2}\{(C_{11}{}^1-C_{12}{}^2-1) C_{22}{}^1+C_{12}{}^1 (C_{22}{}^2-C_{21}{}^1)\}\,.
\end{array}\end{equation}\end{definition}

\begin{definition}\label{D1.3}\rm
The following groups will play an important role in our analysis:
\begin{eqnarray*}
&&\mathcal{G}:=\left\{T:(x^1,x^2)\rightarrow(mx^1,ax^1+bx^2+d)\text{ for }m>0\text{ and }b\ne0\right\},\\
&&\mathcal{H}:=\{T:(x^1,x^2)\rightarrow(mx^1,mx^2+d)\text{ for }m>0\}\subset\mathcal{G},\\
&&\mathcal{I}:=\left\{T:(x^1,x^2)\rightarrow(x^1,ax^1+bx^2)\text{ for }b\ne0\right\}\subset\mathcal{G},\\
&&\mathcal{I}^+:=\left\{T:(x^1,x^2)\rightarrow(x^1,ax^1+bx^2)\text{ for }b>0\right\}\subset\mathcal{I}\,.
\end{eqnarray*}
The Lie group $\mathcal{G}$ is the group of affine transformations which preserve $\mathbb{R}^+\times\mathbb{R}$. The group $\mathcal{G}$ is generated by the
subgroups $\mathcal{H}$ and $\mathcal{I}$; $\mathcal{I}^+$ is the connected component of the identity in $\mathcal{I}$.
If $T\in\mathfrak{H}$, then $T$ acts by homotheties and translations; such a $T$ preserves any Type~$\mathcal{B}$ connection.
Let $\nabla^e$ be the flat Euclidean connection on $\mathbb{R}^+\times\mathbb{R}$ with vanishing Christoffel symbols; $\nabla^C=\nabla^e+(x^1)^{-1}C$.
 Let $T\in\mathcal{I}$. 
Since $T\nabla^e=\nabla^eT$ and $T^*(x^1)=x^1$,  $T^*(\nabla^C)=\nabla^{T^*C}$ where $T^*C$ is defined by the usual linear action
on $\mathcal{W}_{\mathcal{B}}(2)$. More specifically,
$$
(T^*C)(\partial_{x^i},\partial_{x^j},dx^k):=C(T\partial_{x^i},T\partial_{x^j},Tdx^k)\,.
$$
\end{definition}

\begin{definition}\rm
If $X$ is a smooth vector field on $M$, let $\Xi_t^X$ be the local flow
defined by $X$. We say that $X$ is an {\it affine Killing vector field} if $\Xi_t^*\nabla=\nabla$ or,
equivalently (see Kobayashi-Nomizu \cite[Chapter VI]{KN63}), the Lie derivative $\mathcal{L}_X(\nabla)$ of $\nabla$ vanishes.
If $C\in\mathcal{W}_{\mathcal{B}}(2)$ and if $P\in\mathbb{R}^+\times\mathbb{R}$,
let $\mathfrak{X}_C(P)$ be the Lie algebra of germs of affine Killing vector fields at $P$ and let
$\mathcal{X}_C(P)$ be the space of germs of diffeomorphisms of $\mathbb{R}^+\times\mathbb{R}$ at $P$ so $\Phi^*\{\nabla^C\}:=\Phi^{-1}\circ\nabla^C\circ\Phi$ is again a connection of
Type~$\mathcal{B}$. Since the geometry is homogeneous, the particular point $P$ which is chosen is irrelevant.
Let $\mathfrak{h}:=\operatorname{Span}_{\mathbb{R}}\{x^1\partial_{x^1}+x^2\partial_{x^2},\partial_{x^2}\}$ be the Lie algebra of $\mathcal{H}$.
 If  $C\in\mathcal{W}_{\mathcal{B}}(2)$, and $P\in\mathbb{R}^+\times\mathbb{R}$, then $\mathcal{H}\subset\mathcal{X}_C(P)$
and $\mathfrak{h}\subset\mathfrak{X}_C(P)$.
\end{definition}

The following is the main result of this paper. It shows the moduli spaces $\mathfrak{W}^+(p,q)$ are real analytic manifolds and determines the
affine Killing vector fields for any $C\in\mathcal{W}(p,q)$.

\begin{theorem}\label{T1.5}
Let $p+q=2$.
\begin{enumerate}
\item If $C\in\mathcal{W}_{\mathcal{B}}(p,q)$ is not of Type~$\mathcal{C}$, then
$\mathcal{X}(C,P)=\mathcal{G}$,  $\mathfrak{X}(P)\}=\mathfrak{h}$, and $\nabla^C$ is not of Type~$\mathcal{A}$.
\item $\mathfrak{W}_{\mathcal{B}}^+(p,q)$ may be identified with $\mathfrak{Z}_{\mathcal{B}}^+(p,q)/\mathcal{I}^+$.
\item $\mathfrak{W}_{\mathcal{B}}^+(p,q)$ has a natural real-analytic structure.
\item   $\mathcal{Z}_{\mathcal{B}}^+(p,q)\rightarrow\mathfrak{Z}_{\mathcal{B}}^+(p,q)\times\mathcal{I}^+$ is a trivial 
$\mathcal{I}^+$ principal bundle.
\end{enumerate}\end{theorem}

Assertion~(1) shows that only the linear action is
important when considering the local isomorphism type of a Type~$\mathcal{C}$ geometry if the symmetric Ricci tensor is assumed non-degenerate.
Examples of \cite{BGGP16a} show that this can fail of $\rho_{s,C}$ is permitted to have rank $1$ even in the torsion free setting; the assumption that $\rho_{s,C}$
is non-singular is essential.  If instead of considering the more general case of 
Type~$\mathcal{B}$ connections with torsion, we restrict to the subset of Christoffel symbols satisfying the symmetry $C_{ij}{}^k=C_{ji}{}^k$,
the same arguments hold and we obtain some (but not all)
of the results of \cite{BGGP16a} in the Type~$\mathcal{B}$ setting using an entirely different approach.
We note that  \cite{BGGP16a} relied on the complete description of $\mathfrak{X}(C,P)$ for all torsion $C$ which was given in \cite{BGGP16}; 
the discussion of \cite{AMK08}
shows that the possible algebras of Killing vector fields in the setting of torsion is vastly more complicated and thus an approach based on the exhaustive
classification of \cite{BGGP16} is unlikely to be successful in the setting of connections with torsion.

Here is a brief guide to the remainder of this paper. Assertion~1 will be proved in Section~\ref{S2}.
 In Section~\ref{S2.1}, we use the linear fractional transformations over the complex numbers 
or over the para-complex numbers to discuss the orientation preserving isometries of the hyperbolic
plane and of the Lorentzian analogue. In Section~\ref{S2.3}, we use the action of $\mathcal{I}^+$ to put the
symmetric Ricci tensor in normal form. We then examine these normal forms to show $\mathcal{X}(C,P)=\mathcal{H}$ in Section~\ref{S2.4} and in Section~\ref{S2.5}.
We use this analysis to establish the remaining assertions of Theorem~\ref{T1.5}~(1) in
Section~\ref{S2.6} and in Section~\ref{S2.7}. 
The second assertion of Theorem~\ref{T1.5} is immediate from the first
assertion. The remaining assertions are established in Section~\ref{S3}. In Section~\ref{S3.4} we discuss corresponding results for the unoriented moduli spaces
$\mathcal{W}_{\mathcal{B}}(p,q)$.
\section{Reducing the structure group}\label{S2}
The following metric tensors will play a central role in our analysis:
\begin{definition}\label{D2.1}\rm
Let
$$
g_\pm:=\frac{dx^1\otimes dx^1\pm dx^2\otimes dx^2}{(x^1)^2}\text{ and }
g_0:=\frac{dx^1\otimes dx^2+dx^2\otimes dx^1}{(x^1)^2}\,.
$$
The associated Levi-Civita connections $\nabla^\pm$ and $\nabla^0$ are of Type~$\mathcal{B}$; their non-zero Christoffel symbols 
and their Ricci tensors are given by
$$\begin{array}{l}
\textstyle\Gamma_{\pm;11}{}^1=\Gamma_{\pm;12}{}^2=\Gamma_{\pm;21}{}^2=-\frac1{x^1},\ \Gamma_{\pm;22}{}^1=\pm\frac1{x^1},\quad
\textstyle\Gamma_{0,11}{}^1=-\frac1{x^1},\\[0.05in]
\rho_\pm:=\frac1{(x^1)^2}\left(\begin{array}{cc}-1&0\\0&\mp1\end{array}\right),\quad\rho_0=\left(\begin{array}{cc}0&0\\0&0\end{array}\right)\,.
\end{array}$$
\end{definition}

Let $\operatorname{SO}(g)$ be the 3-dimensional Lie group of orientation preserving isometries of
the metric $g\in\{g_+,g_-,g_0\}$. In Section~\ref{S2.1}, we study $\operatorname{SO}(g_\pm)$ and in Section~\ref{S2.2}, we study $\operatorname{SO}(g_0)$.
In Section~\ref{S2.3}, we use the action of the group $\mathcal{I}^+$ to normalize the Ricci tensor to be, modulo sign, $g_+$, $g_-$, or $g_0$.
In the remaining sections, we use the results of these two sections to establish the first assertion of Theorem~\ref{T1.5}
by considering these 3 cases seriatim.

\subsection{Isometry groups of the metrics $g_\pm$}\label{S2.1}
Let $(u,v)=(x^2,x^1)$ to put things in a more standard form.
The orientation preserving isometry groups $\operatorname{SO}(g_\pm)$ of the metrics 
$g_\pm:=v^{-2}(du^2\pm dv^2)$ can be expressed in terms of linear fractional transformations. Let
$$
\operatorname{id}:=\left(\begin{array}{ll}1&0\\0&1\end{array}\right),\quad
\iota_-:=\left(\begin{array}{cc}0&1\\-1&0\end{array}\right),\quad
\iota_+:=\left(\begin{array}{cc}0&1\\1&0\end{array}\right).
$$
Note that $\iota_\pm^2=\pm\operatorname{id}$. Let
$\mathbb{C}_{\pm}:=\operatorname{Span}_{\mathbb{R}}\{\operatorname{id},\iota_\pm\}\subset M_2(\mathbb{R})$;
$\mathbb{C}_-$ is isomorphic to the complex numbers and $\mathbb{C}_+$ is isomorphic to the para-complex numbers.
Let $z_\pm:=u+v\iota_\pm\in\mathbb{C}_\pm$. If $z_-\ne0$, then $z_-$ is invertible; $z_+$ is invertible if and only if $u\ne\pm v$. Let
$$
T_{\pm,A}(z_\pm):=\frac{az_\pm+b}{cz_\pm+d}\text{ for }A=
\left(\begin{array}{cc}a&b\\c&d\end{array}\right)\in\operatorname{SL}(2,\mathbb{R})\,.
$$
Note that $T_{\pm,A}$ is not defined on all of $\mathbb{C}_\pm$ but only on the open dense subset where $cz_\pm+d$ is invertible. Let
$$\begin{array}{lll}
\Re\{z_\pm\}:=u,&\Im\{z_\pm\}:=v,&\bar z_\pm:=u-v\iota_\pm,\\[0.05in]
\textstyle dz_\pm:=du\operatorname{id}+ dv\iota_\pm,& d\bar z:=du\operatorname{id}- dv\iota_\pm,&
\operatorname{PSL}(2,\mathbb{R}):=\operatorname{SL}(2,\mathbb{R})/\{\pm\operatorname{id}\},\\[0.05in]
\partial_{z_\pm}:=\frac12(\partial_u\operatorname{id}\pm\iota_\pm\partial_v),&
\partial_{\bar z_\pm}:=\frac12(\partial_u\operatorname{id}\mp\iota_\pm\partial_v).
\end{array}$$
We have $z_\pm\bar z_\pm=(u^2-v^2)\operatorname{id}$.

\begin{lemma}\label{L2.2}
Adopt the notation established above.
\begin{enumerate}
\item  Let $A$ and $B$ belong to $\operatorname{SL}(2,\mathbb{R})$. Then $T_{\pm,A}\circ T_{\pm,B}=T_{\pm,A\circ B}$.
\item Let $w_\pm:=T_{\pm,A}z_\pm$. Then
\begin{enumerate}
\item $\Im(w_\pm)=\Im(z_\pm)(cz_\pm+d)^{-1}(c\bar z_\pm+d)^{-1}$,
\item $dw_\pm=(cz_\pm+d)^{-2}dz_\pm$ and $d\bar w_\pm=(c\bar z_\pm+d)^{-2}d\bar z_\pm$,
\item $\partial_{w_\pm}=(cz_\pm+d)^2\partial_{z_\pm}$ and $\partial_{\bar w_\pm}=(c\bar z_\pm+d)^2\partial_{\bar z_\pm}$.
\end{enumerate}
\item If $A\in\operatorname{SL}(2,\mathbb{R})$,
then $T_{\pm,A}^*g_\pm=g_\pm$.
\item The map $A\rightarrow T_{\pm,A}$ identifies $\operatorname{PSL}(2,\mathbb{R})$ with the group of orientation preserving isometries of $g_\pm$.
\end{enumerate}\end{lemma}

\begin{proof} Although this result appears in the literature (see, for example, \cite{C05}), we give the proof as it is entirely elementary
to keep our discussion as self-contained as possible and to establish notation for further use.
We prove the first assertion by computing:
\begin{eqnarray*}
&&A=\left(\begin{array}{cc}a&b\\c&d\end{array}\right),\quad \tilde A=\left(\begin{array}{cc}\tilde a&\tilde b\\\tilde c&\tilde d\end{array}\right),\quad 
A\tilde A=\left(\begin{array}{cc}a\tilde a+b\tilde c&a\tilde b+b\tilde d\\c\tilde a+d\tilde c&c\tilde b+d\tilde d\end{array}\right),\\
&&(T_{\pm,A}\circ T_{\pm,B})z_\pm=\frac{a\frac{\tilde az_\pm+\tilde b}{\tilde cz_\pm+\tilde d}+b}{c\frac{\tilde az_\pm+\tilde b}{\tilde cz_\pm+\tilde d}+d}
=\frac{a(\tilde az_\pm+\tilde b)+b(\tilde cz_\pm+\tilde d)}{c(\tilde az_\pm+\tilde b)+d(\tilde cz_\pm+\tilde d)}\\
&&\hspace{2.6cm}=\frac{(a\tilde a+b\tilde c)\tilde z_\pm+(a\tilde b+b\tilde d)}{(c\tilde a+d\tilde c)z_\pm+(c\tilde b+d\tilde d)}
=T_{\pm,A\circ B}z_\pm\,.
\end{eqnarray*}
We prove Assertion~(2a) by computing:
$$\begin{array}{ll}
2\Im(w_\pm)\iota_\pm&=T_{\pm,A}z_\pm-T_{A,\pm}\bar z_\pm
=\frac{(az_\pm+b)(c\bar z_\pm+d)-(a\bar z_\pm+b)(cz_\pm+d)}
{(cz_\pm+d)(c\bar z_\pm+d)}\\[0.05in]
&=\frac{(ad-bc)(z_\pm-\bar z_\pm)}{(cz_\pm+d)(c\bar z_\pm+d)}
=2\frac{v}{(cz_\pm+d)(c\bar z_\pm+d)}\iota_\pm\,.
\end{array}$$
We use the quotient rule to prove Assertion~(2b) by noting:
$$
dw_\pm=\frac{a(cz_\pm+d)-(az_\pm+b)c}{(cz_\pm+d)^2}dz_\pm=(cz+d)^{-2}dz_\pm\,.
$$
Assertion~(2c) then follows by duality. Note that
\begin{eqnarray*}
&&dz_\pm\otimes d\bar z_\pm+d\bar z_\pm\otimes d\bar z_\pm\\
&=&(du\operatorname{id}+ dv\iota_\pm)\otimes(du\operatorname{id}- dv\iota_\pm)+(du\operatorname{id}-dv\iota_\pm )\otimes(du\operatorname{id}+dv\iota_\pm )\\
&=&2(du\otimes du\operatorname{id}-dv\otimes dv\iota_\pm^2)=2(du\otimes du\mp dv\otimes dv)\operatorname{id}\,.
\end{eqnarray*}
We use this identity to express
$$
g_\mp\operatorname{id}=\frac{dz_\pm\otimes d\bar z_\pm+d\bar z_\pm\otimes dz_\pm}{2\Im(z_\pm)^2}\,.
$$
Assertion~(3) then follows from the identities of Assertion~(2) since
\begin{eqnarray*}
&&\frac{dw_\pm\otimes d\bar w_\pm+d\bar w_\pm\otimes w_\pm}{\Im(w_\pm)^2}\\
&=&\frac{dz_\pm\otimes d\bar z_\pm+d\bar z_\pm\otimes dz_\pm}{(cz_\pm+d)^2(c\bar z_\pm+d)^2}\cdot\frac{(cz_\pm+d)^2 (c\bar z_\pm+d)^2}{\Im(z_\pm)^2}\\
&=&\frac{dz_\pm\otimes d\bar w_\pm+d\bar x_\pm\otimes z_\pm}{\Im(x_\pm)^2}\,.
\end{eqnarray*}

We recover the subgroup $\mathcal{H}$ of Definition~\ref{D1.3} by considering the transformations:
$$
T_{\pm,A}:z_\pm\rightarrow\left\{\begin{array}{lll}mz_\pm&\text{ if }&A=\left(\begin{array}{cc}\sqrt m&0\\0&(\sqrt m)^{-1}\end{array}\right)\\
z_\pm+d&\text{ if }&A=\left(\begin{array}{cc}1&d\\0&1\end{array}\right)\end{array}\right\}\,.
$$
These act transitively on $\mathbb{R}^+\times\mathbb{R}$ by homotheties and translations. But we also recover additional 1-parameter subgroups:
\begin{eqnarray*}
&&T_{-,\theta}(z_-):=\frac{\cos\theta z_-+\sin\theta}{-\sin\theta z_-+\cos\theta}\in\operatorname{SO}(g_+)\text{ for }
A=\left(\begin{array}{cc}\cos\theta&\sin\theta\\-\sin\theta&\cos\theta\end{array}\right),\\
&&T_{+,t}(z_+):=\frac{\cosh(t)z_++\sinh(t)}{\sinh(t)z_++\cosh(t)z_+}\in\operatorname{SO}(g_-)\text{ for }
A=\left(\begin{array}{cc}\cosh t&\sinh t\\ \sinh t&\cosh t\end{array}\right)\,.
\end{eqnarray*}
Note that
\begin{eqnarray*}
&&T_{-,\theta}(\iota_-)=\frac{\cos\theta\iota_-+\sin\theta}{-\sin\theta\iota_-+\cos\theta}=
\frac{(\cos\theta\iota_-+\sin\theta)(\sin\theta\iota_-+\cos\theta)}{(-\sin\theta\iota_-+\cos\theta)(\sin\theta\iota_-+\cos\theta)}=\iota_-,\\
&&T_{+,t}(\iota_+)=\frac{\cosh t\iota_-+\sinh t}{\sinh t\iota_-+\cosh t}=
\frac{(\cosh t\iota_++\sinh t)(-\sinh t\iota_++\cosh t)}{(\sinh t\iota_++\cosh t)(-\sinh t\iota_++\cosh t)}=\iota_+\,.
\end{eqnarray*}
Thus the transformations $T_{+,\theta}$ and $T_{-,t}$ preserve $\iota_\pm$. Since $dT_\theta(\iota)$ is a rotation subgroup
and $dT_t(\iota)$ is a hyperbolic rotation group on the tangent space of $\mathbb{C}_\pm$ at $\iota_\pm$, we obtain the full group of orientation preserving isometries.
\end{proof}

\begin{remark}\rm
A bit of caution must be observed here since we must regard $g_-$ as defined on $\mathbb{R}^2$ minus the real axis
and the elements of $\operatorname{PSL}(2,\mathbb{R})$ are not defined everywhere but only have dense ranges and domains. This will
play no role in what follows so we suppress this technicality in the interests of brevity and simplicity.
\end{remark}

\subsection{The isometry group of the metric $g_0$}\label{S2.2} Again, we let $(u,v)=(x^2,x^1)$ so $g_0=v^{-2}(du\otimes dv+dv\otimes du)$.
Let $\Phi_c(u,v):=(u,v)/(1-cv))$. 
\begin{lemma}\label{L2.4}
If $T\in\operatorname{SO}(g_0)$, then $T=H\circ\Phi_c$ for some $c$ and for some $H\in\mathcal{H}$.
\end{lemma}

\begin{proof} Consider the intertwining map $U(u,v):=(u,-v^{-1})=(u,w)$; $U$ is idempotent, i.e. $(u,v)=(u,-w^{-1})$. Since $dv=w^{-2}dw=v^2dw$,
$$g_0=\textstyle\frac{du\otimes dv+dv\otimes du}{v^2}=du\otimes dw+dw\otimes du\,.$$
This shows that $g_0$ is flat. If $T\in{SO}(g_0)$, let $T_1:=U\circ T\circ U\in\operatorname{SO}(du\otimes dw+dw\otimes du)$. We then have
 $T_1(u,w)=(mu+d,m^{-1}w+c)$. Consequently, 
\begin{eqnarray*}
T=UT_1U:(u,v)&\rightarrow&(u,-v^{-1})\rightarrow(mu+d,-m^{-1}v^{-1}+c)\\
&\rightarrow&(mu+d,-(-m^{-1}v^{-1}+c)^{-1})\,.
\end{eqnarray*}
We take $c=0$ to recover the group $\mathcal{H}$;  $(u,v)\rightarrow(mu+d,mv)$.
 And working modulo the action of $\mathcal{H}$, we may set $m=1$ and $d=0$. This yields 
 \medbreak\hfill
 $(u,v)\rightarrow(u,-v/(-1+cv)=v/(1-cv)$.\hfill\vphantom{.}
 \end{proof}

\subsection{Normalizing the symmetric Ricci tensor}\label{S2.3}
Following Definition~\ref{D1.3}, let
$T_{a,b}:(x^1,x^2)\rightarrow(x^1,ax^1+bx^2)$ for $b>0$ belong to $\mathcal{I}^+$.  If $C\in\mathcal{W}_{\mathcal{B}}(p,q)$,
let $\rho_{s,C}$ be the associated Ricci tensor and  let $\rho_{s,C,ij}$ be the components of the Ricci tensor. Note that $(x^1)^2\rho_{s,C_{ij}}$ is constant.

\begin{lemma}\label{L2.5}
 Let $C\in\mathcal{W}_{\mathcal{B}}(p,q)$.
\begin{enumerate}
\item  Let $(p,q)\in\{(2,0),(0,2)\}$. If $C\in\mathcal{W}_{\mathcal{B}}(p,q)$, then there
 exists a unique  $(\lambda(C),a(C),b(C))\in\mathbb{R}^3$ so that
$\rho_{s,T_{a(C),b(C)}C}=\lambda(C)g_+$. The function $C\rightarrow(\lambda(C),a(C),b(C))(C)$ is a real analytic map from $\mathcal{W}_{\mathcal{B}}(p,q)$
to $\mathbb{R}^3$.
\item Let $(p,q)=(1,1)$ and let $\mathcal{O}_{\mathcal{B}}(1,1):=\{C\in\mathcal{W}_{\mathcal{B}}(1,1):\rho_{C,22}\ne0\}$.
\begin{enumerate}\item If $C\in\mathcal{O}_{\mathcal{B}}(1,1)$, then there exists a unique  $(\lambda(C),a(C),b(C))\in\mathbb{R}^3$ so that
$\rho_{s,T_{a(C),b(C)}C}=\lambda(C)g_-$. The function $C\rightarrow(\lambda(C),a(C),b(C))(C)$ is a real analytic map from  $\mathcal{O}_{\mathcal{B}}(1,1)$ to
$\mathbb{R}^3$.
\item If $C\in\mathcal{W}_{\mathcal{B}}(1,1)-\mathcal{O}_{\mathcal{B}}(1,1)$, then there exists a unique $\varepsilon(C)=\pm1$ and a unique $(a(C),b(C))\in\mathbb{R}^2$  so that
$\rho_{s,T_{(a(C),b(C))}C}=\varepsilon g_0$. The function $(\varepsilon(C),a(C),b(C))$ is continuous on
$\mathcal{W}_{\mathcal{B}}(1,1)-\mathcal{O}_{\mathcal{B}}(1,1)$.
\end{enumerate}\end{enumerate}\end{lemma}

\begin{proof} Because $T_{a,b}(x^1,x^2)=(x^1,ax^2+bx^2)$ for $b>0$, we have that:
$$\begin{array}{ll}
(T_{a,b})_*(dx^1)=dx^1,&(T_{a,b})_*(dx^2)=adx^1+bdx^2,\\
(T_{a,b})_*(\partial_{x^1})=\partial_{x^1}-ab^{-1}\partial_{x^2},&(T_{a,b})_*(\partial_{x^2})=b^{-1}\partial_{x^2}\,.
\end{array}$$
Suppose that $\rho_{s,C,22}\ne0$. We apply Equation~(\ref{E1.a}) to compute the Ricci tensor.
For the moment,  let $a(C)$ and $b(C)$ be arbitrary. We compute:
$$
\rho_{s,C}((T_{(a(C),b(C))})_*\partial_{x^1},T_{(a(C),b(C))})_*\partial_{x^2})=b(C)^{-1}\{\rho_{s,C,12}-a(C)\rho_{s,C,22}\}\,.
$$
To ensure that $\rho_{s,C,12}=0$, we must therefore take
$$
a(C):=\frac{\rho_{s,C,12}}{\rho_{s,C,22}}\in\mathbb{R}\,.
$$
Thus $a(C)$ is uniquely determined and is a real analytic function of $C$.
We set $C_1(C):=T_{a(C),1}C$. We now examine
\begin{eqnarray*}
&&\rho_{s,C_1}((T_{0,\tilde b})_*\partial_1,(T_{0,\tilde b})_*\partial_1)=\rho_{C_1}(\partial_1,\partial_1),\\
&&\rho_{s,C_1}((T_{0,\tilde b})_*\partial_2,(T_{0,\tilde b})_*\partial_2)=b^{-2}\rho_{C_1}(\partial_2,\partial_2)\,.
\end{eqnarray*}
Suppose first that $(p,q)\in\{(2,0),(0,2)\}$. Then $\rho_{s,C}$ and hence $\rho_{s,C_1}$ are definite so  $\rho_{s,C_1,11}$ and $\rho_{s,C_1,22}$ have the same sign
and are non-zero.
We set 
$$
b_1(C):=\left|\frac{\rho_{s,C_1,22}}{\rho_{s,C_1,11}}\right|^{1/2}\text{ and }C_2:=T_{(0,b_1(C))}C_1\,.
$$
We have that $\rho_{s,C_2}$ is diagonal and $\rho_{s,C_2,11}=\pm\rho_{s,C_2,22}$. Thus $\rho_{s,C_2}=\lambda(C)g_+$ if
$C\in\mathcal{W}_{\mathcal{B}}(p,q)$ for $(p,q)\in\{(2,0),(0,2)\}$ and $\rho_{s,C_2}=\lambda(C)g_-$ if
$C\in\mathcal{O}_{\mathcal{B}}(1,1)$.  We use the group law
to define $(a(C),b(C))$ so that $T_{(a,b)}=T_{(0,b(C_2))}\circ T_{(a(C),1)}$. The parameters are uniquely determined
and vary real analytically with $C$; Assertion~(1) and Assertion~(2a) now follow.

We complete the proof by establishing Assertion~(2b). Suppose that $\rho_{s,C,22}=0$.
We compute
$$\rho_{s,C}(T_{a,1})_*(\partial_{x^1},(T_{a,1})_*\partial_{x^1})=(x^1)^{-2}\{\rho_{s,C,11}-2a\rho_{s,C,12}\}$$
so $a(C)$ is uniquely determined by requiring that $P_{s,C_1,11}=0$. Choosing $b$ appropriately, we can then use $T_{0,b(C_1)}$ to choose $C_2$ so
 $P_{s,C_2,12}=\pm1$ and complete the proof of Assertion~(2b).
 \end{proof}

\subsection{Reduction to the general linear group (Case 1)}\label{S2.4} Let $C$ define a Type~$\mathcal{B}$ structure 
which is not of Type~$\mathcal {C}$ with
$\rho_{s,C,22}\ne0$ and with $\rho_{s,C}$ non-degenerate.
Suppose that $\Phi$ is the germ of a diffeomorphism
so that $\nabla^{\tilde C}:=\Phi^*(\nabla^C)$ is again of Type~$\mathcal{B}$. We musty show  $\Phi\in\mathcal{G}$.  
Let $\nabla^\pm$ be the Levi-Civita connection of the metrics $g_\pm$. Denote the associated Christoffel symbols by $\Gamma_\pm=(x^1)^{-2}C_\pm$;
they are of Type~$\mathcal{B}$ and are given in Definition~\ref{D2.1}.
Pursuant to the discussion in Section~\ref{S2.3}, we can use the action of $\mathcal{I}^+$ to assume 
$\rho_{s,C}=\lambda g_\varepsilon$ and $\rho_{s,\tilde C}=\tilde\lambda g_\varepsilon$ for $\varepsilon=\pm1$. This implies that
$\Phi^*(g_\varepsilon)=\frac{\lambda}{\tilde\lambda}g_\varepsilon$. Since the metric $g_\varepsilon$ is homogeneous and has non-vanishing Gauss curvature,
$g_\varepsilon$ does not admit a non-trivial homothety. Therefore
$\lambda=\tilde\lambda$ and thus $\Phi$ is an isometry of $g_\varepsilon$. By composing with the action $(x^1,x^2)\rightarrow(x^1,-x^2)$ if necessary, we can assume 
$\Phi$ preserves the orientation and thus $\Phi=T_A$ for $A\in\operatorname{SL}(2,\mathbb{R})$ is given by a linear fractional transformation over the complex
numbers or over the para-complex numbers as appropriate, i.e. 
$$
w:=Tz=\frac{az+b}{cz+d}\text{ for }ad-bc=1\,.
$$
Decompose
$\nabla^C=\nabla^\varepsilon+\frac1{x^1}(C-C_\varepsilon)$. The action by pull-back here is the linear action on $C-C_\varepsilon$
since $T$ preserves $\nabla^\varepsilon$. If $X$ and $Y$ are tangent vectors and if $Z^*$ is a cotangent
vector, then
$$
\{T(C-C_\varepsilon)\}(X,Y,Z^*)=\frac{\Im(z)}{\Im(Tz)}(C-C_\varepsilon)(T X,T Y,T Z^*)\,.
$$
We shall suppress the subscripts on $z_\pm$ and $w_\pm$ to simplify the notation and 
 use Lemma~\ref{L2.2} to express objects in the $z$ coordinate system in terms of the $w$ coordinate system
 to express
 $$
 \frac1{\Im(w)}=(cz+d)(c\bar z+d)\frac1{\Im(z)},\quad \partial_z=(cz+d)^{-2}\partial_w,\quad dz=(cz+d)^2dw\,.
 $$
 We may then express
\begin{eqnarray*}
&&\{T^*(C-C_\varepsilon)\}(\partial_z,\partial_z,dz)=(cz+d)^{-1}(c\bar z+d)(C-C_\varepsilon)(\partial_w,\partial_w,dw),\\
&&\{T^*(C-C_\varepsilon)\}(\partial_z,\partial_{\bar z},dz)=(cz+d)(c\bar z+d)^{-1}(C-C_\varepsilon)(\partial_w,\partial_w,dw),\\
&&\{T^*(C-C_\varepsilon)\}(\partial_{\bar z},\partial_z,dz)=(cz+d)(c\bar z+d)^{-1}(C-C_\varepsilon)(\partial_w,\partial_w,dw),\\
&&\{T^*(C-C_\varepsilon)\}(\partial_{\bar z},\partial_{\bar z},dz)=(cz+d)^3(c\bar z+d)^{-3}(C-C_\varepsilon)(\partial_w,\partial_w,dw)\,.
\end{eqnarray*}
The remaining complex Christoffel symbols are given by conjugation.
By hypothesis $T^*(C-C_\varepsilon)$ and $(C-C_\varepsilon)$ are constant and do not vanish identically since $C$ is not Type~$\mathcal{C}$. 
Thus least one of these equations is non-trivial so
$cz+d=\sigma(c\bar z+d)$ for all $z$ and for some $\sigma$; consequently $c=0$. This implies $Tz=(az+b)/d$ and hence $T\in\mathcal{G}$ as desired.
This completes the proof in this special case.

\subsection{Reduction to the general linear group (Case 2)}\label{S2.5}  Let $C$ define a Type~$\mathcal{B}$ structure 
which is not of Type~$\mathcal {C}$ with
$\rho_{s,C,22}=0$ and with $\rho_{s,C}$ is non-degenerate. We can change coordinates to assume $\rho_{s,C}=\lambda g_0$.
Let $\Phi$ be the germ of a diffeomorphism
so that $\nabla^{\tilde C}:=\Phi^*(\nabla^C)$ is again of Type~$\mathcal{B}$. We argue as before to assume that
$\Phi$ is an isometry of $g_0$ and that $\Phi^*(C-C_0)$ is again of Type~$\mathcal{B}$. We apply Lemma~\ref{L2.4} to
assume that $\Phi(u,v)=(u,v/(1-cv))$ where $(u,v)=(x^2,x^1)$.  We must show $c=0$. Let $\varepsilon_{ijk}=\delta_{1,i}+\delta_{1,j}-\delta_{1,k}$
and let $w=v/(1-cv)$. Then
\begin{eqnarray*}
&&\frac1w=(1-cv)\frac1v,\quad dw=(1-cv)^{-2}dv,\quad\partial_w=(1-cv)^2\partial_v,\\
&&(\Phi^*(C-C_0))(X,Y,Z^*)=\frac{v}{w}(C-C_0)(\Phi_*X,\Phi_*Y,\Phi_*Z^*),\\
&&(\Phi^*(C-C_0))_{ij}{}^k=(1-cv)^{1+2\epsilon_{ijk}}(C-C_0)_{ij}{}^k\,.
\end{eqnarray*}
Thus either $C=C_0$ and $C$ is of Type~$\mathcal{C}$ or $c=0$ and $\Phi\in\mathcal{G}$. This completes the proof that
$\mathcal{X}(C,P)=\mathcal{H}$.
\subsection{Affine Killing vector fields and Type~$\mathcal{A}$ geometry}\label{S2.6}
The Lie algebra of affine Killing vector fields is the Lie algebra of the group of germs of diffeomorphisms which preserve $\nabla^C$. Thus
the fact that $\mathfrak{X}(C,P)=\mathfrak{h}$ is the Lie algebra of $\mathcal{H}$ follows from the fact
that $\mathcal{X}(C,P)=\mathcal{H}$.

\subsection{The pull-back of a Type~$\mathcal{B}$ structure}\label{S2.7} 
We complete the proof of Theorem~\ref{T1.5}~(1) by showing that if $\rho_{s,C}$ is non-degenerate, then $C$ is not Type~$\mathcal{A}$. 
Assume, to the contrary, that there exists the germ of a diffeomorphism $\Phi$
so that $\Phi^*(\nabla^C)$ is of Type~$\mathcal{A}$. We use Lemma~\ref{L2.5} to assume $\rho_{s,C}=\lambda g_\varepsilon$ for $\lambda\ne0$
and suitably chosen $\varepsilon$. Since $\Phi^*\nabla^C$ is Type~$\mathcal{A}$,  $\rho_{\Phi^*\nabla^C}$ has constant coefficients and thus is flat. Since 
$\Phi^*(\rho_{\nabla^C})=\rho_{\Phi^*\nabla^C}$, we conclude $\rho_{\nabla^C}$ is flat. Since the metrics $g_\pm$ have non-zero Gauss curvature,
they are not flat. This implies that $\rho_{\nabla^C}=\lambda g_0$. Let $U(u,v):=(u,-v^{-1})$. Then $U^*(du\otimes dv+dv\otimes du)=g_0$ so
$\Phi^*U^*(du\otimes dv+dv\otimes du)=\Phi^*(\lambda g_0)$. The metrics $du\otimes dv+dv\otimes du$ and $\rho_{\Phi^*\nabla^C}$
have constant coefficients. This implies $T:=U\circ\Phi$ is is linear. Since $U$ is idempotent, $\Phi=U\circ T$. Thus $T^*U^*\nabla^C$ is Type~$\mathcal{A}$. Since
linear maps preserve Type~$A$ structures, we may apply $(T^{-1})^*=(T^*)^{-1}$ to see $U^*\nabla^C$ is Type~$\mathcal{A}$. Thus without loss of generality we may assume $\Phi=U$.

Let $\nabla^e$ be the Levi-Civita connection of the hyperbolic metric $du\otimes dv+dv\otimes du$; this is the usual flat connection on Euclidean space.
Let $\nabla^0$ be the Levi-Civita connection of $g_0$ as given in Definition~\ref{D2.1}. 
We then have $U$ intertwines $\nabla^e$ and
$\nabla^0$. We express $\nabla^C=\nabla^0+\frac{C-C_0}v$ and $U^*\nabla^C=\nabla^e+A$ where $A$ is constant. Set $w=-\frac1v$. We compute
$$
\textstyle A_{ij}{}^k=U^*\left(\frac{C-C_0}v\right)_{ij}{}^k=v^{1+2\epsilon_{ijk}}(C-C_0)_{ij}{}^k\,.
$$
This implies $C=C_0$ so $\nabla^C=\nabla^{C_0}$ is the Levi-Civita connection of $g_0$. This is false as $\rho_{\nabla^0}=0$ and $\rho_{\nabla^C}=\lambda g_0$.
Theorem~\ref{T1.5}~(1).

\section{The topology of the moduli spaces}\label{S3}
\subsection{Principal bundles} Let $G$ be a real analytic Lie group which acts in a real analytic fashion
on a real analytic manifold $N$. 
Let $G_P:=\{g\in G:gP=P\}$ be the {\it isotropy group} of the action.  The action is said to be
{\it fixed point free} if $G_P=\{\operatorname{id}\}$ for all $P$. 
The action is said to be {\it proper} if given points $P_n\in N$ and $g_n\in G$ with $P_n\rightarrow P\in N$ and
$g_nP_n\rightarrow\tilde P\in N$, we can
choose a convergent subsequence so $g_{n_k}\rightarrow g\in G$. We refer to \cite{Bo75,GHL04} for the proof of the following result;
see also the discussion in \cite{G16}.

\begin{lemma}\label{L3.1}
Let the action of $G$ on $N$ be fixed point free, proper, and real analytic.
Then there is a natural real analytic structure on the quotient space $N/G$ so that $G\rightarrow N\rightarrow N/G$ is a principal $G$ bundle.
\end{lemma}

\subsection{The action of $\mathfrak{I}^+$ on $\mathcal{W}_{\mathcal{B}}(p,q)$}
We have already shown that we may identify the
moduli space $\mathfrak{W}_{\mathcal{B}}^+(p,q)$ with $\mathcal{W}_{\mathcal{B}}(p,q)/\mathcal{I}^+$. Consequently,
Assertion~2 and Assertion~3 of Theorem~\ref{T1.5} will follow from Lemma~\ref{L3.1} and from the following result.
\begin{lemma}\label{L3.2}
Let $p+q=2$.
\begin{enumerate}
\item $\mathcal{I}^+$ acts without fixed points on $\mathcal{W}_{\mathcal{B}}(p,q)$.
\item The action of $\mathcal{I}$ on $\mathcal{W}_{\mathcal{B}}(p,q)$ is real analytic.
\item The action of $\mathcal{I}$ on $\mathcal{W}_{\mathcal{B}}(p,q)$ is proper.
\end{enumerate}
\end{lemma}

\begin{proof} 
Since there exist a unique $(\lambda,a,b)$ so that $T_{a,b}\rho_{s,C}=\lambda g_\varepsilon$, it follows that the action of $\mathcal{I}^+$ on
$\mathcal{W}_{\mathcal{B}}(p,q)$ is fixed point free and Assertion~1 follows. Assertion~2 is immediate from the definition.

We choose a slightly more convenient parametrization of $\mathcal{I}$ to prove Assertion~3. For $b\ne0$, set
$S_{a,b}(x^1,x^2):=(x^1,b^{-1}(-ax^1+x^2))$. Then:
\begin{equation}\label{E3.a}\begin{array}{ll}
(S_{a,b})_*(dx^1)=dx^1,&(S_{a,b})_*(dx^2)=b^{-1}(-adx^1+dx^2),\\
(S_{a,b})_*(\partial_{x^1})=\partial_{x^1}+a\partial_{x^2},&(S_{a,b})_*(\partial_{x^2})=b\partial_{x^2}\,.
\end{array}\end{equation}
We fix $x^1=1$ and regard $\rho_{s,C}$ as a function of $C$. We then have
\begin{eqnarray*}
&&(S_{a,b}^*\rho_{s,C})_{11}=\rho_{s,C,11}+2a\rho_{s,C,12}+a^2\rho_{s,C,22},\\
&&(S_{a,b}^*\rho_{s,C})_{12}=b(\rho_{s,C,12}+a\rho_{s,C,22}),\quad(S_{a,b}^*\rho_{s,C})_{22}=b^2\rho_{s,C,22}\,.
\end{eqnarray*}

Suppose that
\begin{equation}\label{E3.b}
C_n\rightarrow C_\infty\quad\text{and}\quad\tilde C_n:=S_{a_n,b_n}C_n\rightarrow\tilde C_\infty\,.
\end{equation}
We wish to show there exists a convergent subsequence so $S_{a_n,b_n}\rightarrow S\in\mathcal{I}$. We have equivalently
$$
\tilde C_n\rightarrow\tilde C_\infty\quad\text{and}\quad C_n=S_{a_n,b_n}^{-1}\tilde C_n\rightarrow C_\infty\,.
$$
Clearly there exists a subsequence so $S_{a_n,b_n}\rightarrow S$ if and only if there exists a subsequence so $S_{a_n,b_n}^{-1}\rightarrow S^{-1}$. 
Thus the roles of $C_\star$ and $\tilde C_\star$ are entirely equivalent. We compute:
\begin{eqnarray}
&&\rho_{s,\tilde C_\infty,22}=\lim_{n\rightarrow\infty} b_n^2\rho_{s,C_n,22},\label{E3.c}\\
&&\rho_{s,\tilde C_\infty,12}=\lim_{n\rightarrow\infty}\{b_n(\rho_{s,C_n,12}+a_n\rho_{s,C_n,22})\},\label{E3.d}\\
&&\rho_{s,\tilde C_\infty,11}=\lim_{n\rightarrow\infty}\{\rho_{s,C_n,11}+2a_n\rho_{s,C_n,12}+a_n^2\rho_{s,C_n,22}\}\,.\label{E3.e}
\end{eqnarray}
Note that we can conjugate Equation~(\ref{E3.b}) replacing $\{C_n,C_\infty,\tilde C_n,\tilde C_\infty,S_{a_n,b_n}\}$ by 
$\{\theta C_n,\theta C_\infty,\tilde\theta\tilde C_n,\tilde\theta\tilde C_\infty,\tilde\theta S_{a_n,b_n}\theta^{-1}\}$; $\{S_{a_n,b_n}\}$ admits a convergent subsequence if and only if
$\{\tilde\theta S_{a_n,b_n}\theta^{-1}\}$ admits a convergent subsequence. We choose $\theta$ and $\tilde\theta$ normalize $\rho_{s,C_\infty}$ and $\rho_{s,\tilde C_\infty}$
using Lemma~\ref{L2.5}. We distinguish various cases. 
\subsection*{Case 1: $\rho_{s,\tilde C_\infty,22}\ne0$ and $\rho_{s,C_\infty,22}\ne0$}
Equation~(\ref{E3.c}) implies
$\lim_{n\rightarrow\infty}b_n^2$ exists. Thus by passing to a subsequence if necessary, we may assume that $b_n\rightarrow b\ne0$. 
Examining Equation~(\ref{E3.d}) then implies $\lim_{n\rightarrow\infty}a_n$ exists.

\subsection*{Case 2a. $\rho_{s,C_\infty,22}\ne0$ and $\rho_{s,\tilde C_\infty,22}=0$} We may assume that 
$\rho_{s,C_\infty}=\lambda g_\pm$ and $\rho_{s,\tilde C_\infty}=\tilde\lambda g_0$.
Since $\lim_{n\rightarrow\infty}\rho_{s,C_n,22}=\rho_{s,C_\infty,22}=\lambda\ne0$ and $\rho_{s,\tilde C_\infty,22}=0$,
Equation~(\ref{E3.c}) implies $b_n\rightarrow0$. By Equation~(\ref{E3.d}), $a_n\rightarrow\infty$.
Equation~(\ref{E3.e}) then provides a contradiction as $\rho_{s,C_n,22}\rightarrow\lambda\ne0$ and $\rho_{s,C_n,12}$ and $\rho_{s,C_n,11}$ are bounded.

\subsection*{Case 2b. $\rho_{s,C_\infty,22}=0$ and $\rho_{s,\tilde C_\infty,22}\ne0$} We interchange the roles of $\{C_n,C_\infty\}$
and $\{\tilde C_n,\tilde C_\infty\}$ and use the argument of Case 2a.

\subsection*{Case 3. $\rho_{s,C_\infty,22}=\rho_{s,\tilde C_\infty,22}=0$} By Lemma~\ref{L2.5}, we may assume that
$\rho_{s,C_\infty}=\varepsilon g_0$ and $\rho_{s,\tilde C_\infty}=\tilde\varepsilon g_0$ for
$\varepsilon\in\{\pm1\}\text{ and }\tilde\varepsilon\in\{\pm1\}$.
Let $\Psi(\alpha,\beta;C):\mathbb{R}^2\times\mathcal{W}_{\mathcal{B}}(2)\rightarrow\mathbb{R}^2$ be defined by setting:
\begin{eqnarray*}
&&\Psi(\alpha,\beta;C):=((S_{\alpha,\beta}\rho_{s,C})_{12},(S_{\alpha,\beta}\rho_{s,C})_{11})\\
&&\hspace{1.78cm}=(\beta\rho_{s,C,12}+\alpha\beta\rho_{s,C,22},\rho_{s,C,11}+2\alpha\rho_{s,C,12}+\alpha^2\rho_{s,C,22})\,.
\end{eqnarray*}
We fix $C$ and compute the Jacobian with respect to $(a,b)$:
\begin{eqnarray*}
\det\Psi^\prime&=&\det\left(\begin{array}{cc}\beta\rho_{s,C,22}&\rho_{s,C,12}+\alpha\rho_{s,C,22}\\
2\rho_{s,C,12}+2\alpha\rho_{s,C,22}&0\end{array}\right)\\
&=&-2(\rho_{s,C,12}+\alpha\rho_{s,C,22})^2\,.
\end{eqnarray*}
Suppose $\rho_{s,C_1}=\varepsilon g_0$ for $\varepsilon=\pm1$. Then
$$
\Psi(0,1;C_1)=(\varepsilon,0)\text{ and }\det\Psi^\prime(0,1;C_1)=-2\,.
$$
Let $\epsilon>0$ be given.
The inverse function shows that there exists $\delta=\delta(\epsilon)>0$ so that if $|C-C_1|<\delta$, then there exists a unique $(\alpha,\beta)$ 
with $|(\alpha,\beta-1)|<\epsilon$ so that $\Psi(\alpha,\beta;\rho_{s,C})=(\lambda,0)$, i.e.
$$
(S_{\alpha(C),\beta(C)}\rho_{s,C})_{12}=\varepsilon\text{ and }(S_{\alpha(C),\beta(C)}\rho_{s,C})_{22}=0\text{ for }|C-C_1|<\delta\,.
$$
Assume Equation~(\ref{E3.b}) holds with $\rho_{s,C_\infty}=\varepsilon g_0,$ and $\rho_{s,\tilde C_\infty}=\tilde\varepsilon g_0$.
We may then choose $(\alpha_n,\beta_n)\rightarrow(0,1)$ and $(\tilde\alpha_n,\tilde\beta_n)\rightarrow(0,1)$ so
$$\begin{array}{ll}
(S_{\alpha_n,\beta_n}C_n)_{11}=0,&(S_{\alpha_n,\beta_n}C_n)_{12}=\varepsilon,\\[0.05in]
(S_{\tilde\alpha_n,\tilde\beta_n}\tilde C_n)_{11}=0,&(S_{\tilde\alpha_n,\tilde\beta_n}\tilde C_n)_{12}=\tilde\varepsilon\,.
\end{array}$$
Thus by replacing $\{C_n,C_\infty,\tilde C_n,\tilde C_\infty,S_{a_n,b_n}\}$ by 
$$
\{S_{\alpha_n,\beta_n}C_n,S_{\tilde\alpha_n,\tilde\beta_n}S_{\alpha_n,\beta_n}S_{\alpha_n,\beta_n}^{-1},
C_\infty,S_{\tilde\alpha_n,\tilde\beta_n}\tilde C_n,S_{\tilde\alpha_n,\tilde\beta_n}\tilde C_\infty,S_{a_n,b_n}\}\,,
$$
we may replace Equation~(\ref{E3.b}) by
$$\begin{array}{llll}
C_n\rightarrow C_\infty,&\tilde C_n:=S_{a_n,b_n}C_n\rightarrow\tilde C_\infty,&\rho_{s,C_\infty}=\varepsilon g_0,&\rho_{s,\tilde C_\infty}=\tilde\varepsilon g_0,\\[0.05in](\rho_{s,C_n})_{11}=0,&(\rho_{s,C_{n}})_{12}=\varepsilon,&(\rho_{s,\tilde C_n})_{11}=0,&(\rho_{s,\tilde C_n})_{12}=\tilde\varepsilon\,.
\end{array}$$
We then get the equations:
\begin{eqnarray}
&&0=\lim_{n\rightarrow\infty} b_n^2\rho_{s,C_n,22},\\
&&\tilde\varepsilon=\lim_{n\rightarrow\infty}b_n(\varepsilon+a_n\rho_{s,C_n,22}),\label{E3.g}\\
&&0=\lim_{n\rightarrow\infty}(2a_n\varepsilon+a_n^2\rho_{s,C_n,22})\label{E3.h}\,.
\end{eqnarray}

\subsection*{Case 3a: $\{a_n\}$ is a bounded sequence} We pass to a subsequence to ensure $a_n\rightarrow a$. 
Since $\rho_{s,C_n,22}\rightarrow0$, we use Equation~(\ref{E3.h})
to conclude $a_n\rightarrow0$. Equation~(\ref{E3.g}) then shows $b_n\rightarrow\frac\varepsilon{\tilde\varepsilon}$. This completes the proof in this special case.

\subsection*{Case 3b: $|a_n|\rightarrow\infty$} Because
$\lim_{n\rightarrow\infty}a_n(2\varepsilon+a_n\rho_{s,C_n,22})=0$, we have that $\lim_{n\rightarrow\infty}2\varepsilon+a_n\rho_{s,C_n,22}=0$. We substitute this
into Equation~(\ref{E3.g}) to conclude that $\lim_{n\rightarrow\infty}b_n(\varepsilon-2\varepsilon)=\tilde\varepsilon$. 
Thus $\lim_{n\rightarrow\infty}b_n$ exists and is non-zero. We use
Equation~(\ref{E3.a}) to compute:
\begin{eqnarray}
\tilde C_{12}{}^1&=&\lim_{n\rightarrow\infty}(S_{a_n,b_n}^*C_n)_{12}{}^1=\lim_{n\rightarrow\infty}\{b_nC_{n,12}{}^1+a_nb_nC_{n,22}{}^1\}\label{E3.i},\\
\tilde C_{11}{}^1&=&\lim_{n\rightarrow\infty}(S_{a_n,b_n}^*C_n)_{11}{}^1=\lim_{n\rightarrow\infty}\{C_{n,11}{}^1+2a_nC_{n,12}{}^1+a_n^2C_{n,22}{}^1\}
\label{E3.j},\\
\tilde C_{12}{}^2&=&\lim_{n\rightarrow\infty}(S_{a_n,b_n}^*C_n)_{12}{}^2\nonumber\\
&=&\lim_{n\rightarrow\infty}\{C_{n,12}{}^2+a_nC_{n,22}{}^2-a_n(C_{n,12}{}^1+a_nC_{n,22}{}^1)\}\label{E3.k}
\end{eqnarray}
We examine Equation~(\ref{E3.i}) and the sum of Equation~(\ref{E3.j}) and Equation~(\ref{E3.k}) to see that
$\lim_{n\rightarrow\infty}a_nC_{n,22}{}^1$ and $\lim_{n\rightarrow\infty}a_n(C_{n,12}{}^1+C_{n,22}{}^2)$ exist. Since $|a_n|\rightarrow\infty$,
$\lim_{n\rightarrow\infty}C_{n,22}{}^1=0$ and $\lim_{n\rightarrow\infty}(C_{n,12}{}^1+C_{n,22}{}^2)=0$. We obtain similarly that
$\lim_{n\rightarrow\infty}(C_{n,21}{}^1+C_{n,22}{}^2)=0$. Consequently
$$C_{22}{}^1=0\text{ and }C_{12}{}^1=C_{21}{}^1=-C_{22}{}^2\,.
$$
We impose these relations and compute
$$
(x^1)^2\rho_{s,C,22}=-2(C_{22}{}^2)^2\text{ and }(x^1)^2\rho_{s,C,12}=C_{22}{}^2-C_{21}{}^2C_{22}{}^2\,.
$$
Since $\rho_{s,C,22}=0$ by hypothesis, we have $C_{22}{}^2=0$ and hence $\rho_{s,C,12}=0$ which is false. Thus this case does not in fact arise.
\end{proof}

\subsection{The proof of Theorem~\ref{T1.5}~(4)}
Since $\mathcal{I}^+$ is a contractible Lie group, it follows that any $\mathcal{I}^+$ principal bundle is trivial. The
fact that $\mathcal{W}_{\mathcal{B}}(p,q)\rightarrow\mathfrak{W}_{\mathcal{B}}^+(p,q)$ is a trivial principal bundle 
also follows for $(p,q)\in\{(2,0),(0,2)\}$ from Lemma~\ref{L2.5}.\hfill\qed
\subsection{The unoriented moduli space}\label{S3.4}
Let $\mathcal{I}_C:=\{T_{a,b}\in\mathcal{I}:T_{a,b}C=C\}$ be the isotropy group associated to an element $C\in\mathcal{W}_{\mathcal{B}}(p,q)$. Modulo conjugation,
we can assume $a=0$ and thus if the isotropy group is non-trivial, $C$ is invariant under the coordinate transformation $(x^1,x^2)\rightarrow(x^1,-x^2)$.
This yields the relations $C_{11}{}^2=C_{12}{}^1=C_{21}{}^1=C_{22}{}^2=0$. We compute:
$$
\rho_C=(x^1)^{-2}\left(
\begin{array}{cc}
 (C_{11}{}^1-C_{12}{}^2+1) C_{21}{}^2 & 0 \\
 0 & (C_{11}{}^1-C_{12}{}^2-1) C_{22}{}^1 \\
\end{array}
\right)\,.
$$
We rescale $x^2$ to put $\rho_C$ in diagonal form. This normalizes $C$ and we obtain a
smooth real-analytic 3-dimensional submanifold $\mathfrak{S}$ of $\mathfrak{W}_{\mathcal{B}}^+$. Since $\mathfrak{W}_{\mathcal{B}}^+$ is
6 dimensional, the normal bundle of $\mathfrak{S}$ in $\mathfrak{W}_{\mathcal{B}}^+$ is 3-dimensional and the fiberwise quotient by $\mathbb{Z}_2$
shows that $\mathfrak{W}_{\mathcal{B}}$ is not smooth but has a $\mathbb{Z}_2$ orbifold singularity along the image of $\mathfrak{S}$;
the link being $\mathbb{RP}^2$. By contrast, the corresponding submanifolds of the torsion free moduli space
$\mathfrak{Z}_{\mathcal{B}}^+$ are 2-dimensional since $C_{12}{}^2=C_{21}{}^2$
and since the ambient manifold is 4-dimensional, the quotient is smooth and the projection $\mathfrak{Z}_{\mathcal{B}}^+\rightarrow\mathfrak{Z}_{\mathcal{B}}$
is a ramified $\mathbb{Z}^2$ covering with links $\mathbb{RP}^1=S^1$.

\end{document}